\documentclass[11pt]{amsart}

\usepackage{amssymb, amsmath, amsfonts, amsthm}
\usepackage{amsaddr}
\usepackage{fullpage,lscape}
\usepackage{color}
\usepackage{listings}

\newtheorem{defn}[equation]{Definition}
\newtheorem{lem}[equation]{Lemma}

\newtheorem{thm}[equation]{Theorem}
\newtheorem{cor}[equation]{Corollary}
\newtheorem{remark}[equation]{Remark}

\newcommand{\sg}{\leqslant}

\newcommand{\F}{M}
\newcommand{\lff}{L}
\newcommand{\E}{M}
\newcommand{\lee}{L}

\usepackage[colorlinks=true,citecolor=black,linkcolor=black,urlcolor=blue]{hyperref}
\newcommand{\arxiv}[1]{\href{http://arxiv.org/abs/#1}{\texttt{arXiv:#1}}}

\begin{document}
\title{Hurwitz generation in groups\\ of types $F_4$, $E_6$, $^2E_6$, $E_7$ and $E_8$}
\author{Emilio Pierro}
\address{\textnormal{Department of Mathematics,\\
London School of Economics and Political Science,\\
Houghton Street,\\
London, WC2A 2AE}}

\date{\today}
\maketitle

\begin{abstract}
A Hurwitz generating triple for a group $G$ is an ordered triple of elements $(x,y,z) \in G^3$ where $x^2=y^3=z^7=xyz=1$ and $\langle x,y,z \rangle = G$.
For the finite quasisimple exceptional groups of types $F_4$, $E_6$, $^2E_6$, $E_7$ and $E_8$, we provide restrictions on which conjugacy classes $x$, $y$ and $z$ can belong to if $(x,y,z)$ is a Hurwitz generating triple.
We prove that there exist Hurwitz generating triples for $F_4(3)$, $F_4(5)$, $F_4(7)$, $F_4(8)$, $E_6(3)$ and $E_7(2)$, and that there are no such triples for $F_4(2^{3n-2})$, $F_4(2^{3n-1})$, $E_6(7^{3n-2})$, $E_6(7^{3n-1})$, $SE_6(7^n)$ or $^2E_6(7^n)$ when $n \geq 1$.
\end{abstract}

\section{Introduction}
It was proven by Hurwitz \cite{hurwitz} that if $G$ is a group of orientation-preserving isometries of a compact Riemann surface of genus $g$, then the order of $G$ is bounded above by $84(g-1)$; groups which attain this bound are known as Hurwitz groups.
The question of which groups are Hurwitz can be translated purely into the language of group theory by the following, since the above is equivalent to $G$ being a finite quotient of the Fuchsian group
\[\Delta := \Delta(2,3,7) = \langle x,y,z \mid x^2 = y^3 = z^7 = xyz = 1 \rangle.\]
A Hurwitz triple in a group $G$ is an ordered triple $(x,y,z)$ of elements $x,y,z \in G$ such that $o(x)=2$, $o(y)=3$, $o(z)=7$, $xyz=1$.
A Hurwitz generating triple is a Hurwitz triple in $G$ such that $\langle x,y \rangle = G$.
We say that $G$ is a Hurwitz group if it admits a Hurwitz generating triple.
We refer the reader to two the most recent survey articles \cite{cond2,MR2523424} which survey the landscape of this problem.
Since quotients of $\Delta$ are perfect, a natural reduction is to consider which non-abelian finite quasisimple groups are Hurwitz.
Recall that a group $G$ is quasisimple if $G/Z(G)$ is simple.
By the Classification of Finite Simple Groups, the non-abelian finite simple groups fall into three families: the alternating groups, the finite simple groups of Lie type (divided into the classical and the exceptional groups) and a finite family of sporadic groups.

\smallskip

The determination of which alternating groups are Hurwitz was completed by Conder \cite{condhur}; for the sporadic groups, the problem was completed by Wilson \cite{hurwitzm} with a large contribution by Woldar \cite{sporhur}.
For classical groups, the problem is understandably quite broad and we mention only a handful of results in this area.
For small rank groups, a summary of all Hurwitz groups which are subgroups of $PGL_7(q)$ is given by Pellegrini and Tamburini in \cite{MR3400398}.
For large rank groups, it was shown by Lucchini, Tamburini and Wilson that $SL_n(q)$ is Hurwitz for all $n \geq 287$ and all prime powers $q$ \cite{MR1745399}.
In between, the picture is a lot more patchy and we simply mention the articles of Vsemirnov \cite{MR2118177} and Vincent and Zalesski \cite{MR2291680} which tackle various classical groups in dimension $7 \leq n \leq 287$.

\smallskip

In the case of the exceptional groups, the groups $^2B_2(q)$ have order coprime to $3$ and so are never Hurwitz groups.
The status of all members of the families $^2G_2(q)'$, $^2F_4(q)'$, $G_2(q)'$ and $^3D_4(q)$ is known \cite{rr, malle,mallesmallhurwitz}.
Their proofs are based primarily on structure constant arguments; similar arguments can be used to show that $F_4(2)$ is not a Hurwitz group, as can be found in earlier work of the author \cite{pierroe62} where it is also shown that $E_6(2)$ is not a Hurwitz group.
It was first shown by Norton in unpublished work, also using structure constants, that $^2E_6(2)$ is a Hurwitz group.
We draw the interested reader's attention to the seemingly less well-known work of Tchakerian \cite{MR2952446} where beautiful arguments involving Chevalley generators are used are used to produce explicit Hurwitz generators for the groups $G_2(3^n)$.

\smallskip

This leaves the remaining exceptional groups of types $F_4$, $E_6$, $^2E_6$, $E_7$ and $E_8$ to which we now turn.
Our original motivation for this paper was a result of Larsen, Lubotzky and Marion \cite[Corollary 1.5]{MR3254328} which suggested there may exist an infinite family of groups of type $E_6$ which are not Hurwitz.
We shall also consider the non-simple quasisimple groups $SE_6$, $^2SE_6(q)$ and $SE_7(q)$.
Our approach, as in the aforementioned papers dealing with the classical groups, is to use the following specialisation of a theorem of Scott \cite{scottcoho}.
\begin{thm}[Scott]\label{scott}
Let $\rho \colon G \to GL(n,V)$ be a representation of $G$ and let $d_G^V$ denote the dimension of the fixed point space of $G$ in $V$.
Let $G^*$ denote the dual representation of $G$ on $V^*$, the dual of $V$.
Let $x$, $y$, $z \in G$ be such that $\langle x,y \rangle = G$ and $xyz =1$.
Then
\[d_x^V + d_y^V + d_z^V \leq n+d_G^V+d_{G^*}^V.\]
\end{thm}
In our case, we shall consider both the standard and adjoint representations of the exceptional groups in question.
Our aim is to then determine so-called ``admissible'' triples of conjugacy classes for each group which we define as follows.

\begin{defn}
Let $G$ be a group and let $x,y,z \in G$ such that $\langle x,y \rangle =G$ and $xyz=1_G$.
If there exists a representation $\rho:G \to GL(n,V)$ such that the bound in Theorem \ref{scott} is not satisfied, then we say that $(x,y,z)$ is \emph{not admissible}.
Otherwise, we say that $G$ is \emph{admissible}.
\end{defn}

\smallskip

Our main theorem is the following.

\begin{thm} \label{main}
The groups $F_4(2^{3n-2})$, $F_4(2^{3n-1})$, $E_6(7^{3n-2})$, $E_6(7^{3n-1})$, $SE_6(7^n)$ and ${}^2E_6(7^n)$, where $n \geq 1$, are not Hurwitz groups.
\end{thm}

We are also able to determine the following.

\begin{thm} \label{main}
Let $G$ be isomorphic to one of $F_4(q)$, $E_6(q)$, $SE_6(q)$, $^2E_6(q)$, $^2SE_6(q)$, $E_7(q)$, $SE_7(q)$ or $E_8(q)$.
If there exists an admissible Hurwitz triple $(x,y,z)$ for $G$, then $G$ and the conjugacy classes to which $x$, $y$ and $z$ belong appear in Tables \ref{f4admiss}, \ref{e6admiss}, \ref{e7admiss} and \ref{e8admiss}.
\end{thm}

Having identified admissible Hurwitz triples, in a handful of cases we are able to explicitly find Hurwitz generating triples and we prove the following.

\begin{thm}\label{new}
The groups $F_4(3)$, $F_4(5)$, $F_4(7)$, $F_4(8)$, $E_6(3)$ and $E_7(2)$ are Hurwitz groups.
\end{thm}

Unfortunately, the group $E_6(7^3)$ was much too large for us to perform a random search in, and so we are unable to prove or disprove its status as a Hurwitz group.

\subsection{Notation}
Unless otherwise specified, $x$, $y$ and $z$ will refer to elements of orders 2, 3 and 7 respectively.
Conjugacy classes of unipotent elements will be denoted by their Carter notation as in \cite{lawjor}, from where our data is obtained.
Conjugacy classes of semisimple elements will be denoted by Atlas \cite{atlas} notation i.e. $2A$, $3B$, etc.
We abuse notation and terminology by referring to a conjugacy class when we may mean a family of conjugacy classes as follows.
E.g. if $x$ and $x^k \neq x$ have the same order, but belong to two different conjugacy classes, we shall refer to ``the conjugacy class of $x$'', where for the purpose of our investigation, $x$ and $x^k$ perform the same role.
Where we write for brevity, for example, that $(A_7,3A/3D,7F/7K)$ is an admissible Hurwitz generating triple for $E_7(p^n)$ where $p \neq 2,3,7$, we mean that each of the triples $(A_7,3A,7F)$, $(A_7,3A,7K)$, $(A_7,3D,7F)$ and $(A_7,3D,7K)$ are admissible triples.

\section{Admissible Hurwitz triples in $F_4(q)$}
In this section we determine admissible Hurwitz generating triples for the exceptional groups of type $F_4(q)$.
We consider both a $26$-dimensional minimal representation $M$ and the $52$-dimensional adjoint representation $L$ of $F_4(q)$.
In characteristic $p \geq 5$, both of these representations are irreducible, and hence their fixed point spaces are zero dimensional.
In characteristic 2, the minimal representation is irreducible but the adjoint representation is not: it splits into the direct sum of two non-isomorphic 26-dimensional s, interchanged by an exceptional outer automorphism.
Nevertheless, this representation still has zero dimensional fixed point space.
In characteristic 3, the adjoint representation is irreducible, but now there are two non-isomorphic minimal representations, dual to one another.
The sum of their fixed point spaces is equal to 1, and so in characteristic 3, $n + d_G^M + d_{G^*}^M = 27$.
Otherwise $n + d_G^M + d_{G^*}^M = 26$ and $n + d_G^L + d_{G^*}^L = 52$.
Since the dimension of the fixed point space of an element is the same across non-isomorphic representations, we may abuse notation and refer to ``the'' minimal representation of $G$.

\subsection{Conjugacy classes in $F_4(q)$}
The conjugacy classes of unipotent elements and their Jordan block structure on the minimal and adjoint representations can be found in Tables 3 and 4 in \cite{lawjor} respectively.
From there it is routine to determine the dimensions of their fixed point spaces.
For the conjugacy classes of semisimple elements and their fixed point space dimensions we reproduce the information given in \cite[Table 2]{cowa} which holds for semisimple elements over fields of characteristic $p>0$.
For consistency, we maintain their notation.
We gather all of this information in Table \ref{f4cc}.
For reference, elements from the classes $2A$ and $2B$ are represented by $t_1$ and $t_4$ respectively in \cite[Table 4.5.1]{gls3}; elements from the classes $3A$, $3C$ and $3D$ are represented by $t_1$, $t_2$ and $t_4$ respectively in \cite[Table 4.7.3A]{gls3}.

\begin{table}[h!]\centering
\begin{tabular}{c | c c}
$x$				&$d_x^\F$&$d_x^\lff$\\
\hline
$A_1$			&20		&36\\
$\tilde{A_1}$		&16		&36\\
$\tilde{A_1}^{(2)}$	&16		&31\\
$A_1 + \tilde{A_1}$	&14		&28\\
\hline
$2A$				&14		&24\\
$2B$				&10		&36
\end{tabular}
\quad\quad
\begin{tabular}{c | c c}
$y$				&$d_y^\F$&$d_y^\lff$\\
\hline
$A_1$			&20		&36\\
$\tilde{A_1}$		&16		&30\\
$A_1 + \tilde{A_1}$	&14		&24\\
$A_2$			&14		&22\\
$\tilde{A_2}$		&9		&22\\
$A_2 + \tilde{A_1}$	&10		&18\\
$\tilde{A_2} + A_1$	&9		&18\\
\hline
$3A$				&14		&22\\
$3C$			&8		&16\\
$3D$			&8		&22
\end{tabular}
\quad\quad
\begin{tabular}{c | c c}
$z$				&$d_z^\F$&$d_z^\lff$\\
\hline
$A_1$			&20		&36\\
$\tilde{A_1}$		&16		&30\\
$A_1 + \tilde{A_1}$	&14		&24\\
$A_2$			&14		&22\\
$\tilde{A_2}$		&8		&22\\
$A_2 + \tilde{A_1}$	&10		&18\\
$B_2$			&10		&16\\
$\tilde{A_2} + A_1$	&8		&16\\
$C_3(a_1)$		&8		&14\\
$F_4(a_3)$		&8		&12\\
$B_3$			&8		&10\\
$C_3$			&4		&10\\
$F_4(a_2)$		&4		&8\\
\end{tabular}
\quad\quad
\begin{tabular}{c | c c}
$z$				&$d_z^\F$&$d_z^\lff$\\
\hline
$7A$				&14		&22\\
$7H$			&6		&12\\
$7I$				&8		&12\\
$7J$				&8		&10\\
$7L$				&4		&12\\
$7N$			&2		&10\\
$7O$			&4		&8\\
$7Q$			&8		&22
\end{tabular}\caption{\label{f4cc} Conjugacy classes of elements in $F_4(q)$}
\end{table}

\subsection{Admissible triples}
We are now ready to prove the following.

\begin{lem}
Let $G \cong F_4(2^n)$.
If $(x,y,z)$ is an admissible Hurwitz triple for $G$, then it is of type $(A_1 + \tilde{A_1},3C,7O)$.
Moreover, $n$ is divisible by $3$.
\end{lem}
\begin{proof}
First we consider the restrictions on $x$ given by the adjoint representation.
Since $d_y^L + d_z^L \geq 16 + 8 = 24$, it follows that $d_x^L \leq 52-24=28$ and so $x \in A_1 + \tilde{A_1}$.
Since $d_x^L=28$ in this case, it follows that $d_y^L=16$ and $d_z^L=8$ forcing $y \in 3C$ and $z \in 7O$.
Finally, the class $7O$ only appears when $n$ is divisible by $3$ \cite{shinf42}, completing the proof.
\end{proof}

The following lemma will facilitate the proofs in the remaining characteristics.

\begin{lem}\label{f42ss}
Let $G \cong F_4(p^n)$, where $p \neq 2$.
If $(x,y,z)$ is an admissible Hurwitz triple for $G$, then $x \in 2A$.
\end{lem}
\begin{proof}
We consider the adjoint representation of $G$.
Across all characteristics $p \geq 3$ we have the bounds $d_y^L \geq 16$ and $d_z^L \geq 8$ and so $d_x \leq 52-16-8=28$.
Hence $x \in 2A$.
\end{proof}

We now turn to characteristic 3.

\begin{lem}
Let $G \cong F_4(3^n)$.
If $(x,y,z)$ is an admissible Hurwitz triple for $G$, then it is of type $(2A, A_2 + \tilde{A_1},7N)$ or $(2A,\tilde{A_2} + A_1,7N/7O)$.
\end{lem}
\begin{proof}
By Lemma \ref{f42ss}, $x \in 2A$.
Then, by considering the adjoint representation, $d_y^L \leq 52-24-8 = 20$ and so $y \in A_2 + \tilde{A_1}$ or $\tilde{A_2} + A_1$.
Similarly, $d_z^L \leq 52-24-18=10$ and so $z \in 7J$, $7N$ or $7O$.
We now consider the minimal representation where $d_x^M+d_y^M+d_z^M \leq 27$, since we are in characteristic 3.
If $y \in A_2 + \tilde{A_1}$, then $d_z^M \leq 27-14-10=3$ and so $z \in 7N$.
If $y \in \tilde{A_2} + A_1$, then $d_z^M \leq 27-14-9=4$ and so $z \in 7N$ or $7O$.
\end{proof}

The following lemma also facilitates the proof in the remaining characteristics.

\begin{lem}
Let $G \cong F_4(p^n)$ where $p \geq 5$.
If $(x,y,z)$ is an admissible Hurwitz triple for $G$, then $y \in 3C$.
\end{lem}
\begin{proof}
By Lemma \ref{f42ss} we know that $x \in 2A$.
Across all characteristics we have the bound $d_z^L \geq 8$ and so $d_y^L \leq 52-24-8=20$.
Hence $y \in 3C$.
\end{proof}

\begin{cor}
Let $G \cong F_4(p^n)$ where $p \geq 5$ and let $(x,y,z)$ be an admissible Hurwitz triple for $G$.\begin{enumerate}
\item If $p=7$, then $(x,y,z)$ is of type $(2A,3C,C_3/F_4(a_2))$.
\item If $p \neq 7$, then $(x,y,z)$ is of type $(2A,3C,7L/7N/7O)$.
\end{enumerate}
\end{cor}
\begin{proof}
By the preceding lemmas, $x \in 2A$ and $y \in 3C$ and hence $d_z^L \leq 52-24-16=12$ and $d_z^M \leq 26 - 14-8 = 4$.
We see that the only conjugacy classes of elements of order 7 satisfying these bounds are the unipotent classes $C_3$, $F_4(a_2)$ and the semisimple classes $7L$, $7N$ and $7O$.
\end{proof}

Finally, we summarise the various admissible triples determined in this section according to their characteristic in Table \ref{f4admiss}.
We also note that the classes $7L$, $7N$ and $7O$ only appear when $p^n \equiv \pm1 \mod 7$ \cite{shoji}.

\begin{table}[h!]\centering
\label{admiss}
\begin{tabular}{c | c c c | c c c}
$G$				&$x$						&$y$					&$z$				&Condition\\ \hline
$F_4(2^n)$		&$A_1 + \tilde{A_1}$			&$3C$				&$7O$			&$2^n \equiv \pm 1 \; (7)$\\
\hline
$F_4(3^n)$		&$2A$					&$A_2 + \tilde{A_1}$		&$7N$\\
				&$2A$					&$\tilde{A_2} + A_1$		&$7N$\\
				&$2A$					&$\tilde{A_2} + A_1$		&$7O$			&$3^n \equiv \pm 1 \; (7)$\\
\hline
$F_4(7^n)$		&$2A$					&$3C$				&$C_3, \; F_4(a_2)$\\
\hline
$F_4(p^n)$		&$2A$					&$3C$				&$7N$\\
$p \neq 2,3,7$		&$2A$					&$3C$				&$7L, \; 7O$		&$p^n \equiv \pm 1 \; (7)$
\end{tabular}\caption{\label{f4admiss} Admissible Hurwitz  triples for the groups $F_4(q)$}
\end{table}

\section{Admissible Hurwitz triples in $E_6^\epsilon(q)$ and $SE_6^\epsilon(q)$}
In this section we treat the cases where $G$ is isomorphic to $E_6(q)$, $^2E_6(q)$, $SE_6(q)$ or $^2SE_6(q)$.
As is common, we write $E_6^\epsilon(q)$ where $\epsilon = 1$ designates the untwisted group and $\epsilon=-1$ designates the twisted group.
Similarly, for $SE_6^\epsilon(q)$.
The centre of $SE_6^\epsilon(q)$ has order $d = (3,q-\epsilon)$.

\smallskip

The groups $SE_6^\epsilon(q)$ have a minimal representation $M$ of dimension 27 which is irreducible in all characteristics; the adjoint representation $L$ of $G$ has dimension 78 and is irreducible except in characteristic 3 where $d_G^L + d_{G^*}^L=1$.

\subsection{Conjugacy classes}
We now turn to the conjugacy classes of the various versions of these groups.
For the unipotent case we refer to \cite[Tables 5 \& 6]{lawjor}; for the semisimple elements of $E_6^\epsilon(q)$ we refer to \cite[Tables 4.5.1 \& 4.7.3A]{gls3}, whose notation we refer to as the ``GLS Class'', and for $SE_6^\epsilon(q)$ we refer to \cite[Table 2]{cowa}, whose notation we follow.

\smallskip

In characteristic $p\neq3$ the number of semisimple classes of elements of order 3 varies according to the version of the group and the congruence of $q$ modulo 9.
There are four cases to consider which we describe now and summarise in Table \ref{e6uss3}.
For the remainder of this subsection we let $q \equiv \epsilon \mod 3$.
The cases are then:
\begin{enumerate}
\item $G \cong SE_6^\epsilon(q)$,
\item $G \cong E_6^\epsilon(q)$ where $q \equiv \epsilon \mod 9$,
\item $G \cong E_6^\epsilon(q)$ where $q$ is not congruent to $\epsilon \mod 9$,
\item $G \cong E_6^{-\epsilon}(q) \cong SE_6^{-\epsilon}(q)$.
\end{enumerate}

In case (1) $G$ has centre of order 3 and has five non-central conjugacy classes of elements of order 3: the elements of the class $3B$ are not conjugate to their inverses, so we take one of these classes as their representatives.

In case (2) $G$ is the quotient of a group in case (1) by its centre.
The elements from classes $3A$ and $3B$ in case (1) are fused and we denote the image of this class in $G$ as $3A$.
The class consisting of the images of elements from $3C$ is denoted $3C$ and similarly for the class $3D$.
In addition, there are four conjugacy classes of elements of order $3$ whose preimages in $SE_6^\epsilon(q)$ have order 9.
These elements are not conjugate to their inverses and so we only need to consider two additional classes: $3E$ and $3F$.

In case (3) we again have the images of the classes appearing in case (1) which we again denote $3A$, $3C$ and $3D$.
The additional classes appearing in (2) now belong to Aut$(G) \setminus$Inn$(G)$ \cite[Table 4.7.3A]{gls3} and so there are no additional classes.

In case (4) we have three conjugacy classes of elements of order 3 which we denote $3A$, $3C$ and $3D$ maintaining the consistency with their GLS notation.

\begin{table}[h!]\centering
\begin{tabular}{c | c c | c}
$x$			&$d_x^\E$	&$d_x^\lee$	&GLS\\
\hline
$A_1$		&21			&56\\
$2A_1$		&17			&46\\
$3A_1$		&15			&40\\
\hline
$2A$			&15			&38			&$t_2$\\
$2B$			&11			&46			&$t_1$
\end{tabular}
\quad
\begin{tabular}{c | c c}
$y$				&$d_y^\E$	&$d_y^\lee$\\
\hline
$A_1$			&21			&56\\
$2A_1$			&17			&46\\
$3A_1$			&15			&38\\
$A_2$			&15			&36\\
$A_2 + A_1$		&12			&32\\
$2A_2$			&9			&31\\
$A_2 + 2A_1$		&11			&28\\
$2A_2 + A_1$		&9			&27\\
\end{tabular}
\quad
\begin{tabular}{cc|cc|c|c}
$SE_6$		&			&$E_6$		&			&GLS\\
Class		&$d_y^M$	&Class		&$d_y^L$		&Class		&Condition\\
\hline
$3A$			&15			&$3A$		&36			&$t_4$\\
$3B$			&6			&''			&''			&"\\
$3C$		&9			&$3C$		&24			&$t_3$\\
$3D$		&9			&$3D$		&30			&$t_{1,6}$\\
--			&--			&$3E$		&46			&$t_1^{\pm1}$	&$q \equiv \epsilon \; (9)$\\
--			&--			&$3F$		&28			&$t_6^{\pm1}$	&$q \equiv \epsilon \; (9)$
\end{tabular}
\caption{\label{e6uss3} Conjugacy classes of elements of order $3$ in $E_6^\epsilon(q)$ and $SE_6^\epsilon(q)$, $q \equiv \epsilon \mod 3$}
\end{table}

By contrast, the unipotent conjugacy classes of elements of order 3 and the conjugacy classes of elements of orders 2 and 7 are much easier to describe.
As usual, we reproduce the information for the unipotent classes from \cite[Tables 5 \& 6]{lawjor} and for the semisimple classes of elements of orders 2 and 7 from \cite{cowa}.
Elements from the classes $2A$ and $2B$ have GLS notation $t_2$ and $t_1$ respectively.
These appear in Table \ref{e627}.

\begin{table}[h!]\centering
\begin{tabular}{c | c c || c | c c c c}
$z$				&$d_z^\E$	&$d_z^\lee$	&$z$				&$d_z^\E$	&$d_z^\lee$\\
\hline
$A_1$			&21			&56			&$A_3 + A_1$		&9			&22\\
$2A_1$			&17			&46			&$D_4(a_1)$		&9			&20\\
$3A_1$			&15			&38			&$A_4$			&7			&18\\
$A_2$			&15			&36			&$D_4$			&9			&18\\
$A_2 + A_1$		&12			&32			&$A_4 + A_1$		&6			&16\\
$2A_2$			&9			&30			&$A_5$			&5			&14\\
$A_2 + 2A_1$		&11			&28			&$D_5(a_1)$		&6			&14\\
$A_3$			&11			&26			&$E_6(a_3)$		&5			&12\\
$2A_2 + A_1$		&9			&24\\
\end{tabular}
\quad\quad\quad\quad
\begin{tabular}{c | c c || c | c c c c}
$z$			&$d_z^\E$	&$d_z^\lee$	&$z$			&$d_z^\E$	&$d_z^\lee$\\
\hline
$7A$			&15			&36			&$7J$		&9			&18\\
$7B$			&0			&28			&$7K$		&2			&16\\
$7C$		&5			&26			&$7L$		&5			&16\\
$7D$		&1			&26			&$7M$		&3			&14\\
$7E$			&2			&20			&$7N$		&3			&12\\
$7F$			&4			&20			&$7O$		&5			&12\\
$7G$		&4			&18			&$7P$		&0			&46\\
$7H$		&7			&18			&$7Q$		&9			&30\\
$7I$			&9			&20			&$7R$		&0			&30
\end{tabular}\caption{\label{e627} Conjugacy classes of elements of order 2 and 7 in $E_6^\epsilon(q)$, $\epsilon = \pm1$}
\end{table}

\subsection{Admissible triples}
We begin by determining which classes elements of orders 2 or 7 can belong to.
If a triple can be shown not to be admissible for the simple group, the preimages of this triple are not admissible for the quasisimple group as well.
Care must be taken that the order of an element in the pre-image is the same, but since we are able to rule out the classes $3E$ and $3F$, we see that this is true for all admissible classes.

\begin{lem}\label{e6x}
Let $G \cong SE_6^\epsilon(p^n)$ or $E_6^\epsilon(p^n)$ and let $(x,y,z)$ be an admissible Hurwitz triple for $G$.
\begin{enumerate}
\item If $p=2$, then $(x,y,z)$ is of type $(3A_1,3C,7M/7N)$.
\item If $p \neq 2$, then $x \in 2A$.
\end{enumerate}
\end{lem}
\begin{proof}
Consider the adjoint representation $L$ of $G$.
Across all characteristics $d_y^L \geq 24$ and $d_z^L \geq 12$.
If $p \neq 3$, then $d_x^L \leq 78 - 24 - 12 = 42$ and so $x \in 3A_1$ or $2A$ depending on whether $p=2$ or not.
If $p=3$, then $d_x^L \leq 79 - 24 - 12 = 43$ and so again $x \in 2A$.

Now suppose that $p=2$.
The adjoint representation applies in all cases, giving $d_y^L \leq 78 - 40 - 12 = 26$ and so $y \in 3C$.
If $G \cong E_6^\epsilon(2^n)$, then the preimage of $y$ in $SE_6^\epsilon(2^n)$ also belongs to the class $3C$ and so the bounds obtained from the minimal representation apply to all $G$.
Along with the bounds from the adjoint representation we obtain $d_z^M \leq 27 - 15 - 9 = 3$ and $d_z^L \leq 78 - 40 - 24 = 14$.
It follows then that $z \in 7M$ or $7N$.
Notice that the class $7N$ exists for all $n$, but the class $7M$ only exist when $p^n \equiv \epsilon \mod 7$.
\end{proof}

\begin{lem}\label{e6y}
Let $G \cong SE_6^\epsilon(p^n)$ or $E_6^\epsilon(p^n)$ and let $(x,y,z)$ be an admissible Hurwitz triple for $G$.
\begin{enumerate}
\item If $p=3$, then $(x,y,z)$ is of type $(2A,2A_2+A_1,7M/7N)$.
\item If $p \neq 3$, then $y \in 3C$ or $3F$.
\end{enumerate}
\end{lem}
\begin{proof}
Suppose first that $p=3$, so that $G \cong SE_6(3^n) \cong E_6(3^n)$.
In this case we can utilise both the minimal and adjoint representations.
By Lemma \ref{e6x} we know that $x \in 2A$ and so $d_z^M \leq 27 - 15 - 9 = 3$ and $d_y^L \leq 79 - 38 - 27 = 14$.
This means $z \in 7M$ or $7N$ and so $d_z^M \geq 3$.
Then, $d_y^M \leq 27 - 15 - 3 = 9$ and $d_y^L \leq 79 - 38 - 12 = 29$ and so $y \in 2A_2+A_1$.

Now suppose that $p \neq 3$ and consider the adjoint representation of $G$.
Since $d_x^L \geq 38$ we see that $d_y^L \leq 78-38-12=28$ and so $y \in 3C$ or $3F$.
\end{proof}

Remarkably, we are able to completely rule out the case that $G \cong SE_6(7^n)$ or $^2E_6(7^n)$.

\begin{lem}\label{se67}
If $G \cong SE_6(7^n)$ or $^2SE_6(7^n)$, then $G$ is not a Hurwitz group.
\end{lem}
\begin{proof}
Let $G$ be as in the hypothesis and, for a contradiction, assume that there is an admissible Hurwitz triple for $G$.
By Lemmas \ref{e6x} and \ref{e6y} $x \in 2A$ and $y \in 3C$ or $3F$.
The class $3F$ does not belong to $SE_6(7^n)$.
Letting $7^n \equiv \epsilon \mod 3$ we see that $\epsilon = 1$ for all $n$ and so $^2E_6(7^n) \cong E_6^{-\epsilon}(7^n)$ which again does not contain the class $3F$.
Hence $y \in 3C$.
Since $G \cong SE_6(7^n)$ or $^2SE_6(7^n)$ we can consider the minimal representation of $G$ in either case.
This yields $d_z^M \leq 27 - 15 - 9 = 3$ which is not satisfied by any $z$ in $G$.
Hence no such triple exists and $G$ is not a Hurwitz group.
\end{proof}

Since $^2SE_6(7^n) \cong {}^2E_6(7^n)$ it remains to consider the groups $E_6(7^n)$.

\begin{lem}
Let $G \cong E_6(7^n)$. If there exists an admissible Hurwitz triple for $G$, then it is of type $(2A,3F,E_6(a_3))$ and $n$ is divisible by $3$.
\end{lem}
\begin{proof}
Let $G \cong E_6(7^n)$ and suppose $(x,y,z)$ is an admissible Hurwitz triple for $G$.
By Lemmas \ref{e6x} and \ref{e6y} we know that $x \in 2A$ and $y \in 3C$ or $3F$.
If $y \in 3C$, then the preimages of $x$, $y$ and $z$ in $SE_6(7^n)$ will also be an admissible Hurwitz triple.
This contradicts Lemma \ref{se67} and so $y \in 3F$.
The class $3F$ only exists if $n$ is divisible by 3.
Finally, from the adjoint representation we see that $d_z^L \leq 78 - 38 - 28 = 12$ and so $z \in E_6(a_3)$.
\end{proof}

Finally, we turn to the restrictions of semisimple elements of order $7$.

\begin{lem}
Let $G \cong E_6^\epsilon(p^n)$ or $SE_6^\epsilon(p^n)$ where $p \neq 2,3,7$.
If $(x,y,z)$ is an admissible Hurwitz triple for $G$, then it is of type $(2A,3C,7K/7M/7N)$ or $(2A,3F,7N/7O)$.
\end{lem}
\begin{proof}
Let $G$ be as in the hypothesis and let $(x,y,z)$ be a Hurwitz generating triple for $G$.
By Lemmas \ref{e6x} and \ref{e6y} we know that $x \in 2A$ and $y \in 3C$ or $3F$.

Suppose first that $y \in 3F$ so that $G \cong E_6^\epsilon(p^n)$ and hence we can only utilise the adjoint representation of $G$.
Then $d_z^L \leq 78 - 38 - 28 = 12$ and so $z \in 7N$ or $7O$.
Next, if $y \in 3C$ we can utilise both the minimal and adjoint representations from which we find $d_z^M \leq 27 - 15 - 9 = 3$ and $d_z^L \leq 78 - 38 - 24 = 16$.
This yields $z \in 7K$, $7M$ or $7N$.
Note that the classes $7K$ and $7M$ only exist when $p^n \equiv \epsilon \mod 7$.
\end{proof}

We summarise the results of this section in Table \ref{e6admiss}.

\begin{table}[h!]\centering
\label{admiss}
\begin{tabular}{c | c c c | c c c}
$G$					&$x$			&$y$				&$z$						&Conditions\\ \hline
$E_6^\epsilon(2^n)$		&$3A_1$		&$3C$			&$7N$\\
					&''			&''				&$7M$					&$2^n \equiv \epsilon \; (7)$\\
\hline
$E_6^\epsilon(3^n)$		&$2A$		&$2A_2 + A_1$		&$7N$\\
					&''			&''				&$7M$					&$3^n \equiv \epsilon \; (7)$\\
\hline
$E_6(7^n)$			
					&2A			&$3F$			&$E_6(a_3)$				&$7^n \equiv 1 \; (9)$\\
\hline
$E_6^\epsilon(p^n)$		&$2A$		&$3C$			&$7N$\\
$p \neq 2,3,7$			&"			&"				&$7K, \; 7M$				&$p^n \equiv \epsilon \; (7)$\\
					&''			&$3F$			&$7N$					&$p^n \equiv \epsilon \; (9)$\\
					&''			&"				&$7O$					&$p^n \equiv \epsilon \; (63)$\\
\hline
\hline
$SE_6^\epsilon(2^n)$	&$3A_1$		&$3C$			&$7N$\\
					&''			&''				&$7M$					&$2^n \equiv \epsilon \; (7)$\\
\hline
$SE_6^\epsilon(p^n)$	&$2A$		&$3C$			&$7N$\\
$p \neq 2,3,7$			&"			&"				&$7K, \; 7M$				&$p^n \equiv \epsilon \; (7)$
\end{tabular}\caption{\label{e6admiss} Admissible Hurwitz triples for the groups $E_6(q)$}
\end{table}

\section{Admissible Hurwitz triples in $E_7(q)$ and $SE_7(q)$}
We now turn to the groups $E_7(q)$ and $SE_7(q)$.
The centre of $SE_7(q)$ has centre of order $d=(q-1,2)$.
Throughout this section we let $M$ denote the $56$-dimensional minimal representation of $SE_6(q)$, which is irreducible in all characteristics, and we let $L$ denote the $133$-dimensional adjoint representation $L$ of $E_6(q)$, which is irreducible in all characteristics, except in characteristic $2$ where $d_G^L + d_G^{L^*}=1$.

\subsection{Conjugacy classes}
We begin by determining the conjugacy classes of elements of order 2, 3, and 7 in $E_7(q)$ and $SE_7(q)$.
As in the case of $E_6$, the presence of a non-trivial centre means the conjugacy classes of semisimple involutions require a little care.

\smallskip

In the case of the unipotent elements, we turn again to \cite[Tables 7 \& 8]{lawjor} and present these, along with the dimensions of their fixed points spaces on $M$ and $L$ in Tables \ref{e72}, \ref{e73} and \ref{e7uni7}.
For the conjugacy classes of semisimple elements our references are \cite[Tables 4.5.1, 4.5.2 \& 4.7.3A]{gls3} and \cite[Table 7]{cohgri}.
We mostly use the notation of \cite{cohgri} but include the ``GLS Class'', referring to \cite[Table 4.5.1]{gls3}, in Tables \ref{e72}, \ref{e73} and \ref{e7ss7}.

\smallskip

We briefly describe the conjugacy classes of involutions in odd characteristic; the number of which will differ according to whether $|Z(G)|=1$ or $2$ \cite[Table 4.5.1 and 4.5.2]{gls3}.
In the case that $G \cong SE_7(q)$, there are two non-central conjugacy classes of involutions which we denote $2A$ and $2B$, both of which project onto the class $2A$ in $G/Z(G)$.
Following \cite{cohgri} we distinguish the classes by letting $2A$ consist of elements $x$ such that $d_x^M = 32$.
If $G \cong E_7(q)$, then there are three conjugacy classes of involutions.
To avoid confusion with the class $2B$, we refer to them as $2A$, $2C$ and $2D$.
Their representatives are denoted $t_1$, $t_4$ (or $t_4'$) and $t_7$ (or $t_7'$) respectively in \cite{gls3}.

\begin{table}[h!]
\begin{tabular}{c|cccc}
$x$		&$d_x^M$	&$d_x^L$\\ \hline
$A_1$	&44			&99\\
$2A_1$	&36			&81\\
$(3A_1)''$	&28			&80\\
$(3A_1)'$	&32			&71\\
$4A_1$	&28			&70
\end{tabular}
\quad\quad\quad\quad\quad\quad
\begin{tabular}{cc|cc|cc}
$SE_7$		&			&$E_7$		&			&GLS\\
Class		&$d_x^M$	&Class		&$d_x^L$		&Class\\\hline
$2A$			&32			&$2A$		&69			&$t_1$\\
$2B$			&24			&''			&''			&''\\
--			&--			&$2C$		&63			&$t_4$ or $t_4'$\\
--			&--			&$2D$		&79			&$t_7$ or $t_7'$	
\end{tabular}\caption{\label{e72} Unipotent and semisimple conjugacy classes of involutions in $E_7(q)$ and $SE_7(q)$}
\end{table}%
\begin{table}[h!]\centering
\begin{tabular}{c|cc||c|cc}
$y$			&$d_y^M$	&$d_x^L$		&$y$			&$d_y^M$	&$d_y^L$\\\hline
$A_1$		&44			&99			&$A_2+A_1$	&26			&57\\
$2A_1$		&36			&81			&$A_2+2A_1$	&24			&51\\
$(3A_1)''$		&27			&79			&$2A_2$		&20			&49\\
$(3A_1)'$		&32			&69			&$A_2+3A_1$	&21			&49\\
$A_2$		&32			&67			&$2A_2+A_1$	&20			&45\\
$4A_1$		&27			&63
\end{tabular}
\quad\quad\quad\quad\quad\quad
\begin{tabular}{c|cc|cc}
$y$			&$d_y^M$	&$d_y^L$	&GLS\\\hline
$3A$			&$14$		&$49$	&$t_4$\\
$3B$			&$2$			&$79$	&$t_7$\\
$3C$		&$20$		&$43$	&$t_2$\\
$3D$		&$20$		&$49$	&$t_6$\\
$3E$			&$32$		&$67$	&$t_1$
\end{tabular}\caption{\label{e73} Unipotent and semisimple conjugacy classes of elements of order $3$ in $E_7(q)$}
\end{table}%
\begin{table}[h!]
\begin{tabular}{c|cc||c|cc||c|ccccc}
$z$			&$d_z^M$	&$d_z^L$		&$z$				&$d_z^M$	&$d_z^L$		&$z$				&$d_z^M$	&$d_z^L$\\
\hline
$A_1$		&44			&99			&$(A_3+A_1)''$		&17			&47			&$D_4+A_1$		&15			&31\\
$2A_1$		&36			&81			&$2A_2+A_1$		&20			&43			&$A_4+A_1$		&14			&29\\
$(3A_1)''$		&27			&79			&$(A_3+A_1)'$		&20			&41			&$D_5(a_1)$		&14			&27\\
$(3A_1)'$		&32			&69			&$D_4(a_1)$		&20			&39			&$A_4+A_2$		&12			&27\\
$A_2$		&32			&67			&$A_3+2A_1$		&17			&39			&$(A_5)'$			&12			&25\\
$4A_1$		&27			&63			&$D_4$			&20			&37			&$A_5+A_1$		&9			&25\\
$A_2+A_1$	&26			&57			&$D_4(a_1)+A_1$	&16			&37			&$D_5(a_1)+A_1$	&12			&25\\
$A_2+2A_1$	&24			&51			&$A_3+A_2$		&16			&35			&$D_6(a_2)$		&9			&23\\
$A_3$		&34			&49			&$A_4$			&16			&33			&$E_6(a_3)$		&12			&23\\
$2A_2$		&20			&49			&$A_3+A_2+A_1$	&15			&33			&$E_7(a_5)$		&9			&21\\
$A_2+3A_1$	&21			&49			&$(A_5)''$			&9			&31			&$A_6$			&8			&19
\end{tabular}\caption{\label{e7uni7} Conjugacy classes of elements of order $7$ in $E_7(7^n)$}
\end{table}%

\begin{table}[h!]\centering
\begin{tabular}{c|cc||c|cc||c|cc||c|cc}
$z$		&$d_z^M$&$d_z^L$		&$z$		&$d_z^M$&$d_z^L$		&$z$		&$d_z^M$&$d_z^L$		&$z$		&$d_z^M$&$d_z^L$\\
\hline
$7A$		&$12$	&$37$		&$7G$	&$4$		&$29$		&$7M$	&$12$	&$27$		&$7S$	&$0$		&$79$\\
$7B$		&$0$		&$49$		&$7H$	&$6$		&$25$		&$7N$	&$2$		&$31$		&$7T$	&$20$	&$49$\\
$7C$	&$10$	&$29$		&$7I$	&$0$		&$33$		&$7O$	&$8$		&$19$		&$7U$	&$16$	&$33$\\
$7D$	&$10$	&$27$		&$7J$	&$2$		&$29$		&$7P$	&$0$		&$47$		&$7V$	&$32$	&$67$\\
$7E$		&$0$		&$37$		&$7K$	&$6$		&$21$		&$7Q$	&$12$	&$23$		&$7W$	&$20$	&$21$\\
$7F$		&$8$		&$21$		&$7L$	&$2$		&$47$		&$7R$	&$20$	&$39$
\end{tabular}\caption{\label{e7ss7} Semisimple conjugacy classes of elements of order $7$ in $E_7(q)$}
\end{table}%

\subsection{Admissible triples}
We first turn to determining admissible Hurwitz generating triples in characteristic $2$.

\begin{lem}
Let $G \cong E_7(2^n)$ and suppose that $(x,y,z)$ is an admissible Hurwitz generating triple.
Then $(x,y,z)$ is of type $(4A_1,3C,7F/7K/7O)$.
\end{lem}
\begin{proof}
Suppose that $G \cong E_7(2^n)$.
By first considering the adjoint representation we see that $d_y^L \leq 134 - 70 - 19 = 45$ and so $y \in 3C$.
Then, since $d_z^L \geq 19$ we see that $d_x^L \leq 134 - 43 - 19 = 72$ and so $x \in (3A_1)'$ or $4A_1$.
Next we have $d_z^L \leq 134 - 70 - 43 = 21$ and so $z \in 7F$, $7K$, $7O$ or $7W$.
We now consider the minimal representation and find that $d_z^M \leq 56 - 28 - 20 = 8$ and so $z \notin 7W$.
Moreover, $d_z^M \geq 6$ and so $d_x^M \leq 56 - 20 - 6 = 30$ and so $x \notin (3A_1)'$.
Hence $x \in 4A_1$.
\end{proof}

Next we consider characteristic $3$.
\begin{lem}
If $G \cong SE_7(3^n)$ and $(x,y,z)$ is an admissible Hurwitz triple for $G$, then $(x,y,z)$ is of type $(2B,2A_2+A_1,7O)$.
\end{lem}
\begin{proof}
Let $G$ and $(x,y,z)$ be as in the hypothesis.
By considering the adjoint representation we see that $d_y^L \leq 133 - 69 - 19 = 45$ and so $y \in 2A_2+A_1$.
Similarly, $d_z^L \leq 19$ and so $z \in 7O$.
By then considering the minimal representation we see that $d_x^M \leq 56 - 20 - 8 = 28$ and so $x \in 2B$.
\end{proof}

\begin{lem}\label{se73}
If $G \cong E_7(3^n)$ and $(x,y,z)$ is an admissible Hurwitz triple for $G$, then $(x,y,z)$ appears in Table \ref{e7admiss}.
\end{lem}
\begin{proof}
We let $G$ and $(x,y,z)$ be as in the hypothesis and consider the adjoint representation of $G$.
Since $d_x^L \leq 133 - 45 - 19$, $x \in 2A$ or $2C$.
If $x \in 2A$, then $d_y^L=45$ and $d_z^L=19$, hence $y \in 2A_2+A_1$ and $z \in 7O$.
Otherwise, $x \in 2B$ in which case $d_y^L + d_z^L \leq 133 - 63 = 70$.
The pairs $(y,z)$ which satisfy this bound appear in Table \ref{e7admiss}.
\end{proof}

In characteristic $7$ we have the following.

\begin{lem}\label{se77}
Let $G$ be a group and let $(x,y,z)$ be an admissible Hurwitz triple for $G$.
\begin{enumerate}
\item If $G \cong SE_7(7^n)$, then $(x,y,z)$ is of type $(2B,3C,E_7(a_5)/A_6)$.
\item If $G \cong E_7(7^n)$, then $(x,y,z)$ appears in Table \ref{e7admiss}.
\end{enumerate}
\end{lem}
\begin{proof}
Let $G$ be a group and let $(x,y,z)$ be a Hurwitz generating triple for $G$.
We first consider the adjoint representation since it applies to both cases.
If $d_x^L = 69$, then $d_y^L \leq 133 - 69 - 19 = 45$ and so $y \in 3C$.
In addition, $d_z^L \leq 133 - 69 - 43 = 21$ and so $z \in E_7(a_5)$ or $A_6$.
If $G \cong SE_7(7^n)$, then we show $x \notin 2A$ as usual, but cannot improve the restriction on $z$.
This proves (1).

Now suppose that $G \cong E_7(7^n)$.
By the above argument, we see that $(2A,3C,E_7(a_5)/A_6)$ is an admissible triple.
Now suppose that $x \notin 2A$ in which case $x \in 2C$.
Then $d_y^L \leq 133 - 63 - 19 = 51$ and so $y \in 3A$, $3C$ or $3D$.
If $y \in 3A$ or $3D$, then $d_y^L = 49$ and $d_z^L \leq 133 - 63 - 49 = 21$ and again $z \in E_7(a_5)/A_6$
Otherwise $y \in 3C$ and so $d_z^L \leq 133 - 63 - 43 = 27$ and $z$ belongs to one of: $D_5(a_1)$, $A_4+A_2$, $(A_5)'$, $A_5+A_1$, $D_5(a_1)+A_1$, $D_6(a_2)$, $E_6(a_3)$, $E_7(a_5)$ or $A_6$.
This completes the proof.
\end{proof}

Finally we turn to the general case.

\begin{lem}
Let $G$ be a group and let $(x,y,z)$ be an admissible Hurwitz triple for $G$.
\begin{enumerate}
\item If $G \cong SE_7(p^n)$, where $p \neq 2,3,7$, then $(x,y,z)$ is of type $(2B,3C,7F/7K/7O)$.
\item If $G \cong E_7(p^n)$, where $p \neq 2,3,7$, then $(x,y,z)$ appears in Table \ref{e7admiss}.
\end{enumerate}
\end{lem}
\begin{proof}
The proof follows the same procedure as that of Lemma \ref{se77}.
In the case $d_x^L = 69$, it follows that $y \in 3C$, $d_z^L \leq 21$ and so $z \in 7F$, $7K$, $7O$ or $7W$.
If $G \cong SE_7(p^n)$ the minimal representation can again be used to show $x \notin 2A$ and $z \notin 7W$.
If $G \cong E_7(p^n)$ we again have $x \in 2A$ and $z \in 7F$, $7K$, $7O$ or $7W$.

Otherwise, $G \cong E_7(p^n)$, $x \in 2C$, $y \in 3A$, $3C$ or $3D$.
If $y \in 3A$ or $3D$, then $d_y^L = 49$ and so $d_z^L \leq 21$ and again $z \in 7F$, $7K$, $7O$ or $7W$.
If $y \in 3C$, then $d_y^L = 43$, $d_z^L \leq 27$ and so $z \in 7D$, $7F$, $7H$, $7K$, $7M$, $7O$, $7Q$ or $7W$.
This completes the proof.
\end{proof}

\begin{table}[h!]\centering
\begin{tabular}{c | c c c | c c c c}
$G$				&$x$			&$y$								&$z$						&Conditions\\\hline
$E_7(2^n)$		&$4A_1$		&$3C$							&$7O$\\
				&"			&"								&$7F, \; 7K$				&$2^n \equiv \pm1 \; (7)$\\
\hline
$E_7(3^n)$		&$2A$		&$2A_2+A_1$						&$7O$\\
				&$2C$		&$A_2+2A_1$						&"\\
				&''			&$2A_2, \; A_2+3A_1, \; 2A_2+A_1$		&$7O, \; 7W$\\
				&''			&"								&$7F, \; 7K$				&$3^n \equiv \pm1 \; (7)$\\
				&''			&$2A_2+A_1$						&$7F, \; 7H, \; 7K$			&$3^n \equiv \pm1 \; (7)$\\
\hline
$E_7(7^n)$		&$2A, \; 2C$	&$3C$							&$E_7(a_5), \; A_6$\\
				&$2C$		&$3A, \; 3D$						&"\\
				&"			&$3C$							&$D_5(a_1), \; A_4+A_2, \; A_5+A_1$\\
				&"			&"								&$D_5(a_1)+A_1, \; D_6(a_2), \; E_6(a_3)$\\
\hline
$E_7(p^n)$		&$2A, \; 2C$	&$3C$							&$7O, \; 7W$\\
$p \neq 2,3,7$		&"			&"								&$7F, \; 7K$				&$p^n \equiv \pm1 \; (7)$\\
				&$2C$		&"								&$7D, \; 7H, \; 7M, \; 7Q$		&$p^n \equiv \pm1 \; (7)$\\
				&"			&$3A, \; 3D$						&$7O, \; 7W$\\
				&"			&"								&$7F, \; 7K$				&$p^n \equiv \pm1 \; (7)$\\
\hline
\hline
$SE_7(3^n)$		&$2B$		&$2A_2+A_1$						&$7O$\\
\hline
$SE_7(7^n)$		&$2B$		&$3C$							&$E_7(a_5), \; A_6$\\
\hline
$SE_7(p^n)$		&$2B$		&$3C$							&$7O$\\
$p \neq 2,3,7$		&"			&"								&$7F, \; 7K$				&$p^n \equiv \pm1 \; (7)$
\end{tabular}\caption{\label{e7admiss}Admissible Hurwitz triples for $E_7(q)$}
\end{table}

\section{Admissible Hurwitz triples in $E_8(q)$}
We at last turn to the groups $E_8(q)$.
The Schur multiplier of $E_8(q)$ is trivial for all $q$, and we only consider the smallest representation, which is the adjoint representation, $L$ of dimension 248.
This representation is irreducible in all characteristics.

\smallskip

For the conjugacy classes of unipotent elements and the dimensions of their fixed point spaces we refer to \cite[Table 9]{lawjor} and follow the notation there; for the semisimple elements we refer to \cite[Table 4]{cohgri} and follow the notation there.
We reproduce this data in Tables \ref{e823} and \ref{e87}.

\begin{table}[h!]\centering
\begin{tabular}{c | c | c c}
$x$			&$d_x^L$		&GLS\\
\hline
$A_1$		&190\\
$2A_1$		&156\\
$3A_1$		&138\\
$4A_1$		&128\\
\hline
$2A$			&136			&$t_8$\\
$2B$			&120			&$t_1$
\end{tabular}
\quad\quad\quad
\begin{tabular}{c|c||c|c}
$y$			&$d_y^L$		&$y$				&$d_y^L$\\ \hline
$A_1$		&190			&$A_2+2A_1$		&102\\
$2A_1$		&156			&$A_2+3A_1$		&94\\
$3A_1$		&136			&$2A_2$			&92\\
$A_2$		&134			&$2A_2+A_1$		&88\\
$4A_1$		&120			&$2A_2+2A_1$	&84\\
$A_2+A_1$	&112	
\end{tabular}
\quad\quad\quad
\begin{tabular}{c|c|c}
$y$		&$d_y^L$		&GLS\\
\hline
$3A$		&80			&$t_4$\\
$3B$		&86			&$t_7$\\
$3C$	&92			&$t_1$\\
$3D$	&134			&$t_8$\\
\end{tabular}
\caption{\label{e823} Conjugacy classes of elements of orders 2 and 3 in $E_8(q)$}
\end{table}

\begin{table}[h!]\centering
\begin{tabular}{c|c||c|c||c|cccccc}
$z$			&$d_z^L$	&$z$				&$d_z^L$	&$z$				&$d_z^L$\\ \hline
$A_1$		&190		&$2A_2+2A_1$	&80		&$D_5(a_1)+A_1$	&52\\
$2A_1$		&156		&$A_3+2A_1$		&76		&$A_4+A_2+A_1$	&52\\
$3A_1$		&136		&$D_4(a_1)+A_1$	&72		&$D_4+A_2$		&50\\
$A_2$		&84		&$A_3+A_2$		&70		&$E_6(a_3)$		&50\\
$4A_1$		&120		&$A_4$			&68		&$A_4+A_3$		&48\\
$A_2+A_1$	&112		&$A_3+A_2+A_1$	&66		&$A_5+A_1$		&46\\
$A_2+2A_1$	&102		&$D_4+A_1$		&64		&$D_5(a_1)+A_2$	&46\\
$A_3$		&100		&$D_4(a_1)+A_12$	&64		&$D_6(a_2)$		&44\\
$A_2+3A_1$	&94		&$A_4+A_1$		&60		&$E_6(a_3)+A_1$	&44\\
$2A_2$		&92		&$2A_3$			&60		&$E_7(a_5)$		&42\\
$2A_2+A_1$	&86		&$D_5(a_1)$		&58		&$E_8(a_7)$		&40\\
$A_3+A_1$	&84		&$A_4+2A_1$		&56		&$A_6$			&38\\
$D_4(a_1)$	&82		&$A_4+A_2$		&54		&$A_6+A_1$		&36\\
$D_4$		&80		&$A_5$			&52
\end{tabular}
\quad\quad\quad\quad\quad
\begin{tabular}{c|ccc||c|ccccccc}
$z$		&$d_z^L$\\
\hline
$7A$		&$64$\\
$7B$		&$52$\\
$7C$	&$50$\\
$7D$	&$40$\\
$7E$		&$40$\\
$7F$		&$36$\\
$7G$	&$54$\\
$7H$	&$38$\\
$7I$		&$50$\\
$7J$		&$82$\\
$7K$		&$92$\\
$7L$		&$68$\\
$7M$	&$134$\\
$7N$	&$80$\\
\end{tabular}\caption{\label{e87} Conjugacy classes of elements of orders $7$ in $E_8(q)$}
\end{table}%

We begin with the following.

\begin{lem}\label{e82}
Let $G \cong E_8(2^n)$ and suppose $(x,y,z)$ is an admissible Hurwitz triple for $G$.
Then $(x,y,z)$ is of type $(4A_1,3A,7D/7E/7F/7H)$.
\end{lem}
\begin{proof}
Let $G$ and $(x,y,z)$ be as in the hypothesis.
Since $d_y^L + d_z^L \geq 80 + 36$ in characteristic 2, it follows that $d_x^L \leq 248 - 80 - 36 = 132$ and so $x \in 4A_1$.
Then, since $d_y^L \leq 248 - 128 - 36 = 84$, $y \in 3A$.
Finally, those classes for which $d_z^L \leq 248 - 128 - 80 = 40$ are: $7D$, $7E$, $7F$ and $7H$.
This completes the proof.
\end{proof}

Next we prove the following:

\begin{lem}\label{e8odd}
Let $G \cong E_8(q)$, where $q$ is odd, and suppose $(x,y,z)$ is an admissible Hurwitz triple for $G$.
Then $x \in 2B$, $y \notin 3D$ and $d_y^L + d_z^L \leq 128$.
\end{lem}
\begin{proof}
Let $G$ and $(x,y,z)$ be as in the hypothesis.
Across all characteristics $p \neq 2$ we see that $d_y^L \geq 80$ and $d_z^L \geq 36$ so $d_x^L \leq 248 - 80 - 36 = 132$ and so $x \in 2B$.
Since $d_x^L = 120$, it follows that for an admissible Hurwitz triple, $d_y^L + d_z^L \leq 248 - 120 = 128$.
Since $d_z^L \geq 36$ we see that $y \notin 3D$.
\end{proof}

The precise combinations of conjugacy classes which yield admissible Hurwitz triples for $E_8(q)$ when $q$ is odd are then immediate.
we summarise them in the following corollary.

\begin{cor}
Let $G \cong E_8(p^n)$ where $p \neq 2$ and suppose $(2B,y,z)$ is an admissible Hurwitz triple for $G$.
One of the following holds:
\begin{enumerate}
\item $p=3$, $y \in 2A_2+2A_1$ or $2A_2+A_1$ and $z \in 7D$, $7E$, $7F$ or $7H$;
\item $p=3$, $y \in 2A_2$ and $z \in 7F$;
\item $p=7$, $y \in 3A$ and $z \in A_4+A_3$, $A_5+A_1$, $D_5(a_1)+A_2$, $D_6(a_2)$, $E_6(a_3)+A_1$, $E_7(a_5)$, $E_8(a_7)$, $A_6$ or $A_6+A_1$;
\item $p=7$, $y \in 3B$ and $z \in E_7(a_5)$, $E_8(a_7)$, $A_6$ or $A_6+A_1$;
\item $p=7$, $y \in 3C$ and $z \in A_6+A_1$;
\item $p \neq 3,7$, $y \in 3A$ or $3B$ and $z \in 7D$, $7E$, $7F$ or $7H$;
\item $p \neq 3,7$, $y \in 3C$ and $z \in 7F$.
\end{enumerate}
\end{cor}
\begin{proof}
Let $G$, $y$ and $z$ be as in the hypothesis.
By Lemma \ref{e8odd} we know that $d_y^L + d_z^L \leq 128$.
Those pairs which satisfy this inequality can easily be determined as follows.
First suppose $p=3$.
If $y \in 2A_2+2A_1$ then $d_z^L \leq 44$; if $y \in 2A_2+A_1$ then $d_z^L \leq 40$; and if $y \in 2A_2$ then $d_z^L \leq 36$.
Now suppose $p \neq 3$.
If $y \in 3A$, then $d_z^L \leq 48$; if $y \in 3B$, then $d_z^L \leq 42$; and if $y \in 3C$, then $d_z^L \leq 36$.
Those classes $z$ which satisfy the above bounds appear as stated.
\end{proof}

\begin{table}[h!]\centering
\begin{tabular}{c | c c c | c c c c}
$G$				&$x$			&$y$							&$z$							&Conditions\\
\hline
$E_8(2^n)$		&$4A_1$		&$3A$						&$7H$\\
				&''			&''							&$7D, \; 7E, \; 7F$				&$2^n \equiv 1 \; (7)$\\
\hline
$E_8(3^n)$		&$2B$		&$2A_2+A_1, \; 2A_2+2A_1$		&$7H$\\
				&''			&''							&$7D, \; 7E, \; 7F$				&$3^n \equiv 1 \; (7)$\\
				&''			&$2A_2$						&$7F$						&$3^n \equiv 1 \; (7)$\\
\hline
$E_8(7^n)$		&$2B$		&$3A$						&$A_4+A_3, \; A_5+A_1$\\
				&''			&''							&$D_5(a_1)+A_2, \; D_6(a_2)$\\
				&''			&''							&$E_6(a_3)+A_1$\\
				&''			&$3A, \; 3B$					&$E_7(a_5), \; E_8(a_7), \; A_6$\\
				&''			&$3A, \; 3B, \; 3C$				&$A_6+A_1$\\
\hline
$E_8(p^n)$		&$2B$		&$3A, \; 3B$					&$7H$\\
$p\neq2,3,7$		&"			&"							&$7D, \; 7E, \; 7F$				&$p^n \equiv \pm 1 \; (7)$\\
				&''			&$3C$						&$7F$						&$p^n \equiv \pm 1 \; (7)$
\end{tabular}
\caption{\label{e8admiss}Admissible Hurwitz triples for $E_8(q)$}
\end{table}

\section{New examples of Hurwitz groups}
Using the results of the previous sections, we turn to Magma \cite{magma} to search for explicit Hurwitz generating triples for groups of a tractable size; unfortunately $E_8(2)$ is out of reach for us, its order is approximately $2^{248}$, whereas the largest group we consider, $F_4(8)$, has order approximately $2^{156}$.

\smallskip

In general, we do not wish to search through the entire group and so our strategy exploits the theory of $(B,N)$-pairs.
Roughly speaking, we construct a suitable subgroup of $B$, as a group of upper triangular matrices of $G$, and then choose an element from $N$ which does not belong to a maximal parabolic subgroup of $G$ containing $B$.

\smallskip

Concretely, we employ the explicit generators given in \cite{howlett} for the minimal representations of the above groups.
We begin with a Sylow $p$-subgroup $S$ consisting of upper triangular matrices, where $p$ is the defining characteristic.
If $p$ is even, we let $S_0=S$; if $p$ is odd, we let $S_0$ be the subgroup generated by $S$ along with an admissible involution from $N_G(S)$.

\smallskip

Next, as it turns out, a suitable power of the explicit Coxeter elements $n$ given in \cite{howlett} belongs to an admissible class of elements of order 3, which we denote as usual by $y$.
We then search randomly through admissible involutions in $S_0$ until we find an $x$ such that $o(xy)=7$ and $(x,y,z)$ is an admissible triple.

\smallskip

To prove generation in a na\"ive way, i.e. ask Magma if our elements generate the whole group, is an incredibly expensive operation and in general does not produce an answer (in any reasonable amount of time).
On the contrary, to show that our subgroup $H = \langle x,y\rangle$ acts irreducibly is very cheap.
We thus construct irreducible subgroups of the appropriate groups, and then use results from the literature to prove that $H$ must in fact be $G$.
The complete classification of maximal subgroups of the groups treated in this paper has recently been completed by David Craven.
However, we shall use the results of Liebeck and Seitz \cite{lieseitzirr} since this produces a much shorter list of groups to consider.
Our goal is then to prove Theorem \ref{new}.

\subsection{Proof of Theorem \ref{new}}
We begin by determining for the groups in question the maximal subgroups which act irreducibly on the minimal representation.
In the following lemmas, for a group $M$, we denote by $F^*(M)$ the generalized Fitting subgroup of $M$.
Recall that $F^*(G) = F(G)E(G)$ where $F(G)$ is the Fitting subgroup (the product of all nilpotent normal subgroups) and $E(G)$ the subgroup generated by the components of $G$, where a component is a quasisimple subnormal subgroup of $G$.
The salient point is that $M/F^*(M)$ is soluble and so if a perfect group, such as a Hurwitz group, is contained in $M$, then it is contained in $F^*(M)$.

\smallskip

We begin with a few preliminary lemmas, first in the case where $G$ is of type $F_4$.

\begin{lem} \label{magaard}
Let $G \cong F_4(p)$, where $p$ is odd and suppose that $M$ is a maximal subgroup of $G$ that acts irreducibly on the minimal representation of $G$.
Then one of the following holds:
\begin{enumerate}
\item $p=7$ and $M \cong G_2(7)$;
\item $F^*(M)$ is isomorphic to one of the following:
\[3^3.L_3(3), \quad L_2(25), \quad L_2(27), \quad L_3(3), \quad {}^3D_4(2), \quad U_3(3).\]
\end{enumerate}
\end{lem}
\begin{proof}
See \cite[Corollary 2]{lieseitzirr}.
\end{proof}

\begin{cor}\label{f4odd5}
Let $G \cong F_4(p)$, where $p$ is odd, and let $H$ be a Hurwitz subgroup of $G$.
If $H$ acts irreducibly on the minimal representation and has order divisible by $5$, then $H \cong G$.
\end{cor}
\begin{proof}
Let $G$ and $H$ be as in the hypothesis.
Since $H$ acts irreducibly on the minimal representation, it is contained in a subgroup $M$ isomorphic to one of those listed in the previous Lemma.
Since $5$ divides $|M|$, $M$ must be $L_2(25)$, which is not Hurwitz since $7$ does not divide $|M|$.
Hence $H=G$, completing the proof.
\end{proof}

\begin{lem}\label{f4873109}
Let $G \cong F_4(8)$ and let $H$ be a Hurwitz subgroup of $G$.
If $H$ acts irreducibly on the minimal representation of $G$ and both $73$ and $109$ divide $|H|$, then $H=G$.
\end{lem}
\begin{proof}
Let $G \cong F_4(8)$, let $H$ be any perfect subgroup of $G$ and suppose that $H$ acts irreducibly on the minimal representation of $G$.
By \cite[Corollary 2]{lieseitzirr}, $H$ is contained in one of the following: $F_4(2)$, $^2F_4(8)$, $^2F_4(2)'$, $P\Omega_8^+(8)$, $\Omega_9(8)$, $3^3.L_3(3)$, $L_2(25)$, $L_2(27)$, $L_3(3)$, $A_9$, $A_{10}$ or $L_4(3)$.
Since none of these groups have order divisible by both 73 and 109, it follows that $H = G$.
\end{proof}

Eventually we shall use an explicit construction of $F_4(q)$ and across all $q$ that we consider, we are able to use a fixed representative for $y$.
In the following lemma we demonstrate that for all $q$, $y$ belongs to an admissible class of elements of order $3$.
We let $E_{i,j}$ denote the matrix whose entries are 0 everywhere except for a 1 in the $(i,j)$-th position.

\begin{lem}\label{f4y}
Let $G \cong F_4(p^f)$ with generators as given in \cite[Section 3.5]{howlett} so that $G \sg GL_{26}(p^f)$.
The element $y := n^4$, where
\[n:=E_{1,5}+E_{4,7}+E_{5,9}+E_{6,10}+E_{7,12}+E_{9,20}+E_{11,21}+E_{13,14}+E_{18,23}+E_{19,16}\]
\[+E_{20,24}+E_{21,18}+E_{22,17}+E_{23,19}+E_{24,25}+E_{25,26}+E_{26,22}\]
\[-E_{2,1}-E_{3,2}-E_{8,3}-E_{10,4}-E_{12,11}-E_{14,13}\]
\[-E_{14,14}-E_{15,6}-E_{16,15}-E_{17,8},\]
has order $3$.
If $p=3$, then $y \in \tilde{A_2}+A_1$, otherwise $y \in 3C$.
\end{lem}
\begin{proof}
Following the construction of $G$ in \cite[Section 3.5]{howlett} the element $n$ is shown to be a Coxeter element of $G$, hence $o(n)=12$.
This is a 26-dimensional minimal representation for $G$ (except when $p=3$) and we can then easily compute $d_M^y=9$.
Since $n$ is a Coxeter element of $G$, and all such elements are conjugate in $G$, it is sufficient to find a Coxeter element in an adjoint representation of $G$ and determine the corresponding $d_L^y$.
We can construct such a Coxeter element $n_L$ in the adjoint representation of $G$ in Magma, and then compute $d_y^L$.
If $q=3$, then $d_y^L=18$, otherwise $d_y^L = 16$.
This determines the conjugacy of $y$, as claimed.
\end{proof}

The corresponding lemmas can be proved similarly by direct computation for $E_6(3)$ and $E_7(2)$.

\begin{lem}\label{e6y}
Let $G \cong E_6(3)$ with generators as given in \cite[Section 3.4]{howlett} so that $G \sg GL_{27}(3)$.
The element $y := n^4$, where $n$ is a Coxeter element of $G$, belongs to $2A_2 + A_1$.
\end{lem}

\begin{proof}
By checking both the minimal and adjoint representations, this can easily be verified.
\end{proof}

\begin{lem}\label{e7y}
Let $G \cong E_7(2)$ with generators as given in \cite[Section 3.3]{howlett} so that $G \sg GL_{56}(2)$.
The element $y := n^6$, where $n$ is a Coxeter element of $G$, belongs to $3C$.
\end{lem}

\begin{proof}
The proof is as in the previous lemmas.
\end{proof}

Next we prove the remaining analogous results for $E_6(3)$ and $E_7(2)$.

\begin{lem}\label{e63H}
Let $G \cong E_6(3)$ and suppose that $H$ is a Hurwitz subgroup of $G$.
If $H$ acts irreducibly on the minimal representation and $757$ divides $|H|$, then $H = G$.
\end{lem}
\begin{proof}
Let $G$ and $H$ be as in hypothesis.
Since $H$ contains an element of order $757 = q^6+q^3+1$, $H$ contains a maximal torus of this order.
Since, in addition, $H$ acts irreducibly, we see from \cite[Corollary 2]{lieseitzirr} and \cite{92lsssubs} that $H$ must be contained in a subgroup isomorphic to $L_3(27)$.
Since this is not a Hurwitz group \cite{cohhur}, it follows that $H = G$.
\end{proof}

In the case of $E_7(2)$, the classification of its maximal subgroups was determined in \cite{maxe72}, which will aid the proof of the following lemma.

\begin{lem}\label{e72H}
Let $G \cong E_7(2)$ and $H$ be a Hurwitz subgroup of $H$.
If $H$ acts irreducibly on the minimal representation $V$ of $G$ and if $13$ divides $|H|$, then $H = G$.
\end{lem}
\begin{proof}
Let $G$ and $H$ be as in hypothesis.
If $H$ acts irreducibly on $V$, then, by \cite[Corollary 2]{lieseitzirr} and \cite{maxe72} $H$ is contained in a subgroup isomorphic to one of the following:
\[L_8(2).2, \quad U_8(2).2, \quad (L_2(2)^7).L_3(2), \quad 3^7.(2 \times Sp_6(2)).\]
In fact, when $q = 2$, $(L_2(2)^7).L_3(2) \sg 3^7.(2 \times Sp_6(2))$   \cite[Proof of Lemma 2.5]{92lsssubs}.
Since $H$ is perfect, we can assume that $H$ is contained in $L_8(2)$, $U_8(2)$ or $3^7.Sp_6(2)$.
Since 13 divides $|G|$, but does not divide the orders of $L_8(2)$, $U_8(2)$ or $3^7.Sp_6(2)$, the conclusion holds.
\end{proof}

\smallskip

Our strategy for the remainder of the proof is then the following.
Since the groups themselves are quite large, we use the generators given in \cite{howlett} to construct upper triangular Sylow $p$-subgroups $S$ of $G$, where $p$ is the defining characteristic, in Magma.
If $p$ is even we let $S_0 = S$, otherwise we construct $S_0 := S \rtimes \langle h \rangle$ where $h$ is an admissible involution belonging to a maximal torus normalising $S$.
We let $y$ be as given in Lemma \ref{f4y}, \ref{e6y} or \ref{e7y} as appropriate.
We then perform a random search of admissible involutions $x$ in $S_0$ until we find an $x$ such that $o(xy)=7$.
If $p \neq 2$, then, by construction, since $4$ does not divide the order of $S_0$, any such involution will be conjugate to $h$, which makes our search easier.
Finally we construct the subgroup $H := \langle x,y \rangle$ and use the preceding lemmas to prove that $H=G$ by checking that $H$ acts irreducibly and contains elements of the necessary orders, which are both very cheap operations.

\smallskip

The code for this paper, including the generating matrices in Magma format, can be found online at \arxiv{2003.12595}.
This file also includes the generic code (except for $F_4$ in even characteristic, which can easily be modified from the odd characteristic case).

\subsubsection{$G \cong F_4(3)$}
\begin{lem}
There exists a Hurwitz generating triple for $F_4(3)$ of type $(2A,\tilde{A_2} + A_1,7N)$.
\end{lem}
\begin{proof}
Using the generators for $G \cong F_4(3) \sg GL_{26}(3)$ given in \cite[Section 3.5]{howlett} and implemented in Magma we proceed as follows.
We let $y$ be as in Lemma \ref{f4y} so that $y \in \tilde{A_2} + A_1$.
We then generate the subgroup $S \sg G$ as follows:
\[S := \langle x_A(1), \; x_B(1), \; x_C(1), \; x_D(1), \; h_{2B+C,\xi}\rangle\]
where $\xi$ generates $\mathbb{F}_3^\times$.
We can compute $d_G^{h_{2B+C,\xi}} = 14$ and, since the Sylow $2$-subgroup order of $S$ has order 2, all involutions in $S$ are conjugate and belong to $2A$.
Following a random search through $S$-conjugates of $h_{2B+C,\xi}$ we find such an $x$ (provided in the accompanying files) where $o(xy)=7$ and $d_G^{xy} = 2$, so $xy \in 7N$.
Next, we construct $H := \langle x,y\rangle$ and find that it contains elements of order $5$.
Using Parker's Meataxe in Magma, we find $\mathbb{F}_3^{26}H$ has a 25-dimensional irreducible quotient, $H'$ and so, by Corollary \ref{f4odd5}, we conclude that $H = G$.
\end{proof}

\subsubsection{$G \cong F_4(5)$}
\begin{lem}
There exists a Hurwitz generating triple for $F_4(5)$ of type $(2A,3C,7N)$.
\end{lem}
\begin{proof}
Our method is identical to the preceding lemma.
We use the generators for $G \cong F_4(5) \sg GL_{26}(5)$ as in \cite[Section 3.5]{howlett} and let $y$ be as in Lemma \ref{f4y}.
We then construct the following subgroup:
\[S := \langle x_A(1), \; x_B(1), \; x_C(1), \; x_D(1), \; h_{2B+C,\xi}^2\rangle\]
where $\xi$ generates $\mathbb{F}_5^\times$, so that $h_{2B+C,\xi}^2$ belongs to $2A$ and again a Sylow $2$-subgroup of $S$ has order 2.
We then search through random $S$-conjugates $x$ of $h_{2B+C,\xi}^2 $ until we find an $x$ (provided in the accompanying files) such that $o(xy)=7$ and $xy \in 7N$.
Letting $H:=\langle x,y\rangle$ we find that $H$ contains elements of order $5$ and use Parker's Meataxe to determine that $H$ is irreducible.
Then, by Lemma \ref{f4odd5}, $H = G$.
\end{proof}

\subsubsection{$G \cong F_4(7)$}
\begin{lem}
There exists a Hurwitz generating triple for $F_4(7)$ of type $(2A,3C,F_4(a_2))$.
\end{lem}
\begin{proof}
Our proof is similar to those of the preceding lemmas.
Using the construction of $F_4(7) \sg GL_{26}(7)$ given in \cite[Section 3.5]{howlett}, we let $S$ denote the upper triangular Sylow $7$-subgroup of $G$.
The element $x_0 := h_{2B+C,\xi}^3$ belongs to $2A$ and so we let $S_0 := \langle S,x_0$ so that a Sylow $2$-subgroup of $S_0$ has order 2.
We let $y:=n^4$, where $n$ is the Coxeter element of $G$ given in their construction, so that by Lemma \ref{f4y}, $y \in 3C$.
We then search randomly through $S_0$ conjugates of $x_0$ until we find an $x$ (provided in the accompanying files) such that $z:=xy$ has order $7$ and $d_M^z = 4$.
The invariant factors of $z$ can be determined in Magma and we then see that $z \in F_4(a_2)$.
Letting $H:= \langle x,y \rangle$ we are able to find random elements of order $5$ and using Parker's Meataxe we see that $H$ is an irreducible subgroup.
Then, by Corollary \ref{f4odd5}, it follows that $H=G$.
\end{proof}

\subsubsection{$G \cong F_4(8)$}
\begin{lem}
There exists a Hurwitz generating triple for $F_4(8)$ of type $(A_1 + \tilde{A_1},3C,7O)$.
\end{lem}
\begin{proof}
Using the generators in \cite[Section 3.5]{howlett} we construct in Magma the full upper triangular Borel subgroup of $G \cong F_4(8) \sg GL_{26}(8)$:
\[B:= \langle x_A(1), \; x_B(1), \; x_{2B+C}(1), \; x_C(1), \; x_D(1), \; h_{A,\xi}, \; h_{B,\xi}, \; h_{C,\xi}, \; h_{D,\xi} \rangle\]
where $\xi$ is a primitive element of $\mathbb{F}_8^\times$.
Letting $y$ be as in Lemma \ref{f4y} we search through random elements $x$ of order $2$ in $B$ which belong to $A_1 + \tilde{A}_1$, such that $o(xy)=7$ and $d_M^{xy}=4$.
We eventually find such an $x$ (provided in the accompanying files).
We check that $H:=\langle x,y \rangle$ acts irreducibly, which it does, and that $H$ contains elements of order $73$ and $109$, which it does.
Then, by Lemma \ref{f4873109}, $(x,y,xy)$ is a Hurwitz generating triple for $G$ of type $(A_1+\tilde{A}_1,3C,7O)$.
\end{proof}

\begin{remark}
Of the groups appearing in Theorem \ref{new}, $F_4(8)$ is the largest in cardinality having order approximately $2^{156}$.
The cardinality of $E_8(2)$ is approximately $2^{248}$ and so we did not attempt to search for a Hurwitz generating triple in $E_8(2)$.
\end{remark}

\subsubsection{$G \cong E_6(3)$}
\begin{lem}
There exists a Hurwitz generating triple for $E_6(3)$ of type $(2A,2A_2 + A_1,7N)$.
\end{lem}
\begin{proof}
As in the case of the groups of type $F_4$, we take generators for an upper triangular Sylow $3$-subgroup $S$ of $G \sg GL_{27}(3)$ from \cite[Section 3.3]{howlett}.
The involution $x_0:=h_{a,\mu}$, where $\mu$ generates $\mathbb{F}_3^\times$, has eigenvalue 1 with multiplicity 15, and so $x_0 \in 2A$.
We then let $S_0 := \langle S,x_0\rangle$ so that the Sylow $2$-subgroup of $S_0$ has order 2 and $y$ be as in Lemma \ref{e6y} so that $y \in 2A_2+A_1$.
Finally, we search through random $S_0$-conjugates of $x_0$ until we find an $x \sim x_0$ such that $o(xy)=7$ and $d_M^{xy}=3$ (provided in the accompanying files).
We check that $H:=\langle x,y\rangle$ is an irreducible subgroup of $G$ using Parker's Meataxe, and that $H$ contains an element of order $757$, which both hold for this choice of $x$.
Hence $(x,y,xy)$ is a Hurwitz generating triple for $G$, by Lemma \ref{e63H}, and thus $xy \in 7N$.
\end{proof}

\subsubsection{$G \cong E_7(2)$}
\begin{lem}
There exists a Hurwitz generating triple for $E_7(2)$ of type $(4A_1,3C,7O)$.
\end{lem}
\begin{proof}
As before, we construct $S \sg G \sg GL_{56}(2)$, the Sylow $2$-subgroup consisting of upper triangular matrices using the generators in \cite{howlett}.
We set $x:=x_b(1)x_c(1)x_e(1)x_g(1)$ and $y := n^6 \in 3C$.
It can be determined that $x \in 4A_1$ by computing $d_M^x$ and $d_L^x$ or (more satisfyingly) by inspection of the unipotent conjugacy classes of $G$ \cite{mizuno}.
By Lemma \ref{e72}, $y \in 3C$.
Next, we search through $S$-conjugates of $x$ in order to find an $xy$ of order 7 (such an $x$ is provided in the accompanying files).
For the given $x$ we check that $d_M^z = 8$, which is necessary for $z \in 7O$.
It could be the case that $z \in 7F$; however, by checking that $H:=\langle x,y \rangle$ is irreducible and contains an element of order $13$, by Lemma \ref{e72H}, $H=G$ and $(x,y,z)$ is a Hurwitz generating triple (and hence admissible).
Thus $z\in 7O$, completing the proof.
\end{proof}

\section*{Acknowledgement}
Part of this research was conducted at the Perth campus of the University of Western Australia where the author was supported by ARC Grant DP140100416.
The author acknowledges that the Perth campus is situated on Whadjuk Noongar boodja and pays his respects to Noongar Elders past and present.
The author also wishes to thank Alastair Litterick and David Craven for many helpful conversations in preparation of this paper.

\bibliographystyle{amsplain}
\providecommand{\bysame}{\leavevmode\hbox to3em{\hrulefill}\thinspace}
\providecommand{\MR}{\relax\ifhmode\unskip\space\fi MR }
\providecommand{\MRhref}[2]{%
  \href{http://www.ams.org/mathscinet-getitem?mr=#1}{#2}
}
\providecommand{\href}[2]{#2}

\end{document}